\documentclass[psamsfonts]{amsart}
\usepackage[utf8]{inputenc}

\usepackage{amssymb,amsfonts}
\usepackage{enumerate}
\usepackage{mathrsfs}
\usepackage{url}
\usepackage{mathtools}
\usepackage[font=small,labelfont=bf]{caption}
\usepackage[all,arc]{xy}
\usepackage{todonotes}
\usepackage{enumitem}

\usepackage{xcolor}
\usepackage{soul}

\newtheorem{thm}{Theorem}[section]
\newtheorem{cor}[thm]{Corollary}
\newtheorem{prop}[thm]{Proposition}

\newtheorem{lem}[thm]{Lemma}

\theoremstyle{definition}
\newtheorem{defn}[thm]{Definition}

\newtheorem{exmp}[thm]{Example}

\newtheorem{fact}[thm]{Fact}

\newtheorem{NotConv}[thm]{Notation and Conventions}

\theoremstyle{remark}
\newtheorem{rem}[thm]{Remark}

\newtheorem{para}[thm]{}

\makeatletter
\let\c@equation\c@thm
\makeatother
\numberwithin{equation}{section}

\def\Ind{\setbox0=\hbox{$x$}\kern\wd0\hbox to 0pt{\hss$\mid$\hss} \lower.9\ht0\hbox to 0pt{\hss$\smile$\hss}\kern\wd0} 

\def\Notind{\setbox0=\hbox{$x$}\kern\wd0\hbox to 0pt{\mathchardef \nn=12854\hss$\nn$\kern1.4\wd0\hss}\hbox to 0pt{\hss$\mid$\hss}\lower.9\ht0 \hbox to 0pt{\hss$\smile$\hss}\kern\wd0}

\newcommand{\gal}{\mathrm{Gal}}
\newcommand{\cals}{\mathcal{S}}
\newcommand{\call}{\mathcal{L}}
\newcommand\cald{\mathcal{D}}
\newcommand{\calk}{\mathcal{K}}

\newcommand{\si}{\sigma}
\newcommand{\fix}{\mathrm{Fix}}
\newcommand{\inv}{^{-1}}
\newcommand{\tp}{\mathrm{tp}}

\newcommand{\nat}{\mathbb{N}}
\newcommand{\rat}{\mathbb{Q}}
\newcommand{\rest}{{\lower       .25     em      \hbox{$\vert$}}}

\newcommand{\vlabel}{\label}

\title{Measures on Bounded Perfect PAC Fields}

\author{Zo\'e Chatzidakis}
\email{zoe.chatzidakis@imj-prg.fr}

\author{Nicholas Ramsey}
\email{sramsey5@nd.edu}

\thanks{NR was supported by NSF grant DMS-2246992.  }

\date{\today}

\begin{document}

\begin{abstract}
We describe a construction for producing Keisler measures on bounded perfect PAC fields. As a corollary, we deduce that all groups definable in bounded perfect PAC fields, and even in unbounded perfect Frobenius fields, are definably amenable. This work builds on our earlier constructions of measures for $e$-free PAC fields and a related construction due to Will Johnson. 
\end{abstract}

\maketitle
\setcounter{tocdepth}{1}
\tableofcontents

\section{Introduction}

In an earlier paper \cite{ChRa}, we showed how  to construct measures on definable subsets of perfect PAC
fields whose absolute Galois group is  $e$-free profinite or 
$e$-free pro-$p$. In this paper, we extend the result to small
projective absolute
Galois groups, as well as to some unbounded projective Galois groups,
such as the superprojective ones.  The measures we construct are typically
\emph{not} invariant under definable bijection, but we show that
birational isomorphisms preserve them.  Together with
classification results of Hrushovski and Pillay \cite{HrP} on groups definable in
bounded PAC fields, this allows us to deduce that all definable groups
in these fields are definably amenable, meaning they carry a translation
invariant probability measure on their definable subsets. Via an
ultralimit argument, we conclude that all groups  definable in
perfect Frobenius fields are definably amenable as well.  

The question of whether or not groups definable in bounded PAC fields are
definably amenable arose in connection with the broader question of
whether \emph{all} groups in simple theories are definably amenable. The
standard examples of groups with a simple theory tend to be definably
amenable, which can usually be seen from very general considerations.
The class of definably amenable groups contains the finite, stable and
amenable groups, and is closed under ultraproducts, which entails that
groups definable in pseudo-finite fields (which are definably amenable because pseudo-finite),
groups definable in ACFA (which are definably amenable because pseudo-stable \cite{H04}\cite{SV}),
and the extra special $p$-groups (which are definably amenable
both because nilpotent and because pseudo-finite) are definably amenable. However, starting with
Hrushovski's \cite{H05}, bounded PAC fields have served as a major class of
examples of simple theories.  Although pseudo-finite fields are examples
of bounded PAC fields, the class of such fields is considerably broader and it can be
shown that bounded PAC fields which are not pseudo-finite are not
elementarily equivalent to an ultraproduct of stable fields either
{\cite{S}} (see also (5) in Concluding Remarks).  This suggests that in order to show that groups
definable in these fields are definably amenable, {these general considerations do not suffice and }a different, {more specific} sort of argument is required. 

Although we now know there are examples of groups with a simple theory
that are not definably amenable \cite{CHKKMPR}, the issue of definable
amenability for groups in bounded PAC fields is a good test question for
our understanding of the model theory of these fields. Our construction
of translation invariant measures relies on quantifier simplification
results, a description of definable groups, and the model theory of the
inverse system of the absolute Galois groups of these fields. Moreover,
our analysis extends to perfect Frobenius fields, which include some 
unbounded PAC fields.  Unbounded PAC fields never have simple theory \cite{C-unbounded},
but perfect Frobenius fields are important examples within the
broader class of NSOP$_{1}$ theories.  Our measures provide new tools
for the analysis of groups in this setting. One should note that
since any graph is interpretable in some perfect PAC field \cite[Chapter
28, \S\S 7 -- 10]{FJ}, there are
examples of perfect PAC fields with non-definably amenable definable groups.

To construct measures on all bounded perfect PAC fields, we begin by
constructing them on a distinguished subclass, consisting of the perfect
PAC fields whose absolute Galois group is the universal Frattini cover
of a finite group, by adapting  a Markov chain technique introduced into
model theory by Johnson in \cite{Jo}. We then  show that the same result can be extended to bounded
perfect PAC fields and to perfect Frobenius fields, again
 an ultralimit argument.  Finally, we show that if $G$ is a finite group, then sets definable in an existentially closed $G$-field, also have a measure, and that groups definable in those fields are therefore definably amenable.

\section{Preliminaries}

\begin{NotConv} We will always work in a large algebraically closed
  field $\Omega$. If $K$ is a field, we denote by $K^s$ its separable
  closure, by $K^{alg}$ its algebraic closure, and by $\gal(K)$ its
  absolute Galois group, $\gal(K^s/K)$. \\
   Throughout this paper, all homomorphisms between profinite groups will
  be {\em continuous}. If $G$ is a profinite group, then ${\rm Im}(G)$ denotes
  the set of (isomorphism classes of) finite (continuous) images of
  $G$. 
  \end{NotConv}
\begin{defn}
A field $K$ is called \emph{pseudo-algebraically closed (PAC)} if every absolutely irreducible variety defined over $K$ has a $K$-rational point. 
\end{defn}

\begin{defn} \vlabel{proj-EP}
Let $G$ be a profinite group.
\begin{enumerate}
\item We say $G$ is \emph{projective} if whenever $\alpha: G \to A$ and $\beta : B \to A$ are  epimorphisms, where $A$ and $B$ are finite groups, then there is a  homomorphism $\gamma :G \to B$ so that $\beta \circ \gamma = \alpha$ as in the following diagram:
$$
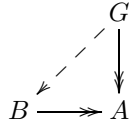
\xymatrix{ & G\ar@{-->}[dl]  \ar@{->>}[d] \\
B \ar@{->>}[r] & A }
$$
\item We say $G$ has the \emph{embedding property} (also known as the
  \emph{Iwasawa property}), if whenever $\alpha: G
  \to A$ and $\beta : B \to A$ are  epimorphisms and $B \in \mathrm{Im}(G)$, then there is a  epimorphism $\gamma :G \to B$ so that $\beta \circ \gamma = \alpha$ as in the following diagram:
$$
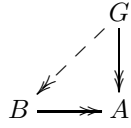
\xymatrix{ & G\ar@{-->>}[dl]  \ar@{->>}[d] \\
B \ar@{->>}[r] & A }
$$
\item The profinite group $G$ is called \emph{superprojective} if it is
  projective and has the embedding property.
  \item The profinite group $G$ is \emph{small} if for every $n$, it has
    only finitely many (continuous) quotients of size $\leq n$. 
  \item The field $K$ is a \emph{Frobenius field} if it is PAC, and
    $\gal(K)$ is superprojective.
    \item The field $K$ is \emph{bounded} if $\gal(K)$ is
      small. Equivalently, if for any $n$, the field $K$ has only finitely many
     Galois extensions of degree $\leq n$. 
\end{enumerate}
\end{defn}

\begin{fact}\vlabel{fact1} \begin{enumerate}
\item Assume that $K\equiv L$ and that $\gal(K)$ is small. Then
  $\gal(L)\simeq \gal(K)$ (Klingen, \cite[Proposition 20.4.5]{FJ}). This
  follows from the following result: if $G$ and $H$ are profinite
  groups, with ${\rm Im}(G)={\rm Im}(H)$ and $G$ small, then $G\simeq
  H$ (\cite[Proposition 16.10.7]{FJ}); note also that if ${\rm
    Im}(H)\subseteq {\rm Im}(G)$, then there is an epimorphism $G\to H$
  (loc. cit.). 
    \item The absolute Galois group of a PAC field is projective (Ax,
  see
  e.g. \cite[Theorem 11.6.2]{FJ}).
\item Moreover, every projective profinite
  group is the absolute Galois group of some PAC field
  (Lubotzky-v.d.Dries, see e.g. \cite[Corollary 23.1.2]{FJ}).

\item The free profinite group on countably many generators $\hat{F}_{\omega}$
has the embedding property so $\omega$-free PAC fields are Frobenius
fields.  However, there are many others\textemdash see, e.g.,
\cite[Sections 24.6 -- 24.9]{FJ}.
\end{enumerate}
\end{fact}


The following is an important tool in our proof.

\begin{defn}\vlabel{Frattini} Let $G$ be a profinite group.
  \begin{enumerate}
\item
   The \emph{Frattini subgroup} $\Phi(G) \leq G$ is the intersection of
   all maximal open subgroups of $G$.  If $G$ is profinite, it is
   pro-nilpotent. 
\item  (\cite[Definition 22.5.1]{FJ}) 
A homomorphism $\varphi: H \to G$ is called a \emph{Frattini cover} if one of the following equivalent conditions holds:
\begin{itemize}
\item[(a)] $\varphi$ is surjective and 
  $\ker(\varphi) \leq \Phi(H)$.
\item [(b)] A closed subgroup $H_{0} \leq H$ is equal to $H$ if and only if $\varphi(H_{0}) = G$.
\item [(c)] A subset $S \subseteq H$ generates $H$ if and only if $\varphi(S)$
  generates $G$.
\end{itemize}
\item A Frattini cover $\alpha: H\to G$ is called \emph{universal} if whenever
  $\beta: H'\to G$ is another Frattini cover of $G$, there exists an
  epimorphism $\gamma:H\to H'$ such that $\beta\circ\gamma =\alpha$, as
  in the following diagram:
  $$
\xymatrix{ & H\ar@{-->>}[dl]  \ar@{->>}[d] \\
H' \ar@{->>}[r] & G }
$$
\end{enumerate}
\end{defn}

\begin{fact}\vlabel{fact2} Let $G$ be a profinite group.
  \begin{enumerate}
    \item
  The universal Frattini cover of $G$ exists, and is denoted $\tilde
  G$ (\cite[Proposition 22.6.1]{FJ}). It is a projective group, and the kernel of the
  map $\tilde G\to G$ is pro-nilpotent, with $\Phi(\tilde G)$ projecting
  onto $\Phi(G)$. For additional properties, see
  \cite[Subsection 22.6]{FJ}.
  \item \vlabel{easy fact} 
  Suppose $\varphi: A \to B$ and $\psi : B
      \to C$ are epimorphisms of profinite groups. Then  $\psi \circ
      \varphi$ is a Frattini cover if and only if $\varphi$ and
        $\psi$ are 
      Frattini covers. 
      \item Observe that if $G$ is finitely generated, and $H\to G$ is a
        Frattini 
        cover, then any system of generators of $G$ lifts to a system of
        generators of $H$. This follows from Definition \ref{Frattini}(c). 

\item \cite[Lemma 22.6.3]{FJ} \vlabel{frat fact}
Let $\tilde{\varphi}: \tilde{G} \to G$ be the universal Frattini cover of a profinite group $G$. Then a profinite group $H$ is a quotient of $\tilde{G}$ if and only if $H$ is a Frattini cover of a quotient of $G$. 
        
  \item Assume that $k_1$ is a Galois extension of the field $k$, and that
    $\gal(k)$ is a Frattini cover of $\gal(k_1/k)$. Let $K$ be a separable 
    extension of $k$. The following
    conditions are equivalent (by easy Galois theory and the definition
    of a Frattini cover):
    \begin{itemize}
    \item[(a)] $K$ is a regular extension of $k$.
      \item[(b)] $\gal(K)$ projects onto $\gal(k_1/k)$ (via the restriction
        map).
      \item[(c)] $K\cap k_1=k$.
        \item[(d)] If $k_0$ is the subfield of $k_1$ fixed by $\Phi(\gal(k_1/k))$,
          then $K\cap k_0=k$.
        \end{itemize}
  \end{enumerate}
  \end{fact}


\subsection*{Quantifier-elimination down to test formulas}
We will write $\mathcal{L}=\{+,-,\cdot,0,1\}$ for the language of rings and, given a field $E$, we write $\mathcal{L}(E)$ for the language $\mathcal{L}$ together with constant symbols for the elements of $E$. 

The following is \cite[Theorem 20.3.3]{FJ}:

\begin{fact} \vlabel{elementary equivalence theorem}
    Let $E/L$ and $F/M$ be separable field extensions in which both $L$ and $M$ contain a common subfield $K$. Suppose $E$ and $F$ are both PAC fields of the same imperfection degree. In addition, suppose there is a $K^{s}$-isomorphism $\Phi_{0}: L^{s} \to M^{s}$ such that $\Phi_{0}(L) = M$ and there is a commutative diagram 
    $$
\xymatrix{ 
\gal(F) \ar@{->}[r]^{\varphi} \ar@{->}[d]^{\mathrm{res}} & \gal(E) \ar@{->}[d]^{\mathrm{res}} \\
\gal(M) \ar@{->}[r]^{\varphi_{0}} & \gal(L),
}
$$
where $\varphi_{0}$ is dual to $\Phi_{0}$. Then $E \equiv_{K} F$. 
\end{fact}

\begin{defn}
Fix a field $E$.
\begin{itemize}
    \item A {\em test sentence over $E$} is a Boolean
combination of $\mathcal{L}(E)$-sentences of the form $\exists y\, f(y)=0$,
where $f(y)\in E[y]$ ($y$ a single variable). 
\item Similarly, we say that $\theta(x)$ is a {\em test formula
   over $E$} (in the tuple of variables $x$)  if it is a Boolean combination of $\mathcal{L}(E)$-formulas of the form $\exists y\, f(x,y)=0$, where $y$ is a
single variable, and $f\in E[x,y]$. 
\end{itemize}
\end{defn}

Note that in particular any quantifier-free
sentence (formula) of $\mathcal{L}(E)$ is a test sentence (formula).

\begin{fact} \cite[Lemma 20.6.3]{FJ}\vlabel{fact:test} Let $E$ and $F$ be extensions of the field $K$, and assume
  that they satisfy the same test sentences over $K$. Then $K^{alg}\cap
  E \simeq_K K^{alg}\cap F$. 
\end{fact}

\begin{lem}\vlabel{lem:test} (Folklore) Let $E$ and $F$ be perfect PAC fields, regular extensions of
  their common subfield $k$, and assume that $\gal(k)$ is small, and that
  $\gal(E)$, $\gal(F)$ are isomorphic to $\gal(k)$. Let $K$ be a common
  subfield of $E$ and $F$ containing $k$. Then $E\equiv_K F$ if and only
  if $E$ and
  $F$ satisfy the same test sentences over $K$.  
\end{lem}

\begin{proof} The left to right implication is clear, so assume that $E$
  and $F$ satisfy the same test sentences over $K$. By Fact
  \ref{fact:test}, there is  $\varphi\in\gal(K)$ such that $\varphi(K^{s}\cap
  E)=K^{s}\cap F$. Note that both these fields have absolute Galois
  group isomorphic to $\gal(k)$. Now apply Fact \ref{elementary
    equivalence theorem} to the dual isomorphism $\Phi:\gal(F\cap
  K^{s})\to \gal(E\cap K^{s})$, $\si\mapsto \varphi\inv\circ\si\circ\varphi$. \end{proof}





    \medskip
Let $k$ be a perfect field, with small Galois group, and fix 
a regular extension $\calk$ of $k$, with $\gal(\calk)\simeq \gal(k)$ via
the restriction map, and
suppose that $\calk$ is perfect PAC. Our results give then
a   description of the types (in $\calk$) over $k$. 

\begin{cor}\vlabel{types} Let $E$ be a subfield of $\mathcal{K}$
  containing $k$, and let $a$ and $b$ be tuples in $\mathcal{K}$. The
  following conditions are equivalent:
  \begin{enumerate}
  \item $\mathrm{tp}(a/E)=\mathrm{tp}(b/E)$;
    \item There is an $E$-isomorphism $E(a)^s\cap \mathcal{K}\to E(b)^s\cap
      \mathcal{K}$ which sends $a$ to $b$;
    \item $a$ and $b$ satisfy the same test-formulas over $E$;
   \item for every finite Galois extension $L$ of $E(a)$, there is a
     field embedding $\varphi:L\to E(b)^s$ such that $\varphi(L\cap
     \mathcal{K})=\varphi(L)\cap \mathcal{K}$.   
 \end{enumerate}
 \end{cor}

\begin{proof} Note that the restriction maps $\gal(\calk)\to
  \gal(E(a)^s\cap\calk)\to \gal(k)$ and $\gal(\calk)\to
  \gal(E(b)^s\cap\calk)\to \gal(k)$ are isomorphisms, and apply Fact \ref{elementary
    equivalence theorem}.  \end{proof}

\begin{lem} \vlabel{test:Frob} (Folklore) Let $k$ be a perfect field, and $\mathcal{K}$ a perfect Frobenius
  field containing $k$ and regular over $k$. Let $a,b\in \calk$. The
  following are equivalent:
  \begin{itemize}
  \item[(a)] $\tp(a/k)=\tp(b/k)$.
    \item[(b)] There is a $k$-isomorphism $k(a)^s\cap\calk\to k(b)^s\cap\calk$
      which sends $a$ to $b$.
      \item[(c)] $a$ and $b$ satisfy the same test-formulas over $k$. 
\end{itemize}
\end{lem}

\begin{proof} We already know that (b) and (c) are equivalent (Fact
  \ref{fact:test}), and clearly (a) implies (c). Assume now that (b)
  holds, and let $\varphi:k(a)^s\cap\calk\to k(b)^s\cap\calk$ be a
  $k$-isomorphism sending $a$ to $b$. Then (an easy adaptation of)
  Theorem 30.6.3 of \cite{FJ} gives the result (with $M=M'=\calk$). 

\end{proof}



\para {\bf Inverse systems associated to profinite groups}. 
We recall
briefly the definition and structure of the complete system $S(G)$ associated to a
profinite group $G$. For more properties and details, see for instance, \cite{C-IP},
and for the ``interpretation'' of $S(\gal(K))$ in $K$, see the appendix of
\cite{C-omega}. 

To the profinite group $G$, we associate the structure $S(G)$ defined as
follows: its universe is the disjoint union of all  $G/N$, where $N$ runs
over all open normal subgroups of $G$. One then knows
that $$G=\varprojlim_N G/N,$$
where for $N\subset M$, the connecting map $\pi_{N,M}$ is the
  natural projection $G/N\to G/M$.
The language is a many-sorted language, with sorts indexed by the
positive integers, and by convention an element $gN$ ($\in G/N$) is of sort $n$
if and only if $[G:N]\leq n$. We let $S(G)_i$  be the set of elements
of $S(G)$ of sort $\leq i$. We then have relational symbols $C,\leq,P$
(which should be indexed by tuples of integers, but for notational simplicity we omit these) and a
constant $1$, which
are interpreted as follows:
$1=G/G$; $gN\leq hM$ if and only if $N\subseteq M$; $C(gN,hM)$ if and only if
$gM=hM$; $P(g_1N_1,g_2N_2,g_3N_3)$ if and only if $N_1=N_2=N_3$ and
$g_1g_2N_1=g_3N_1$. Using $\leq$, one also defines $gN\sim hM$ if and
only if $N=M$. The $\sim$-equivalence classes of $S(G)$ are then the finite
quotients of $G$, they come with their group law (given by $P$), as well as projection
maps (with graphs given by $C$) onto their quotients.  There is a set $\Sigma$ of sentences of
the language 
which axiomatizes the structures of the form $S(G)$ for a profinite group $G$, and an epimorphism
$G\to H$ dualizes to an embedding $S(H)\to S(G)$. 

Note that the set of $\sim$-equivalence classes of $S(G)$ forms a modular
lattice, and that any finite quotient of an $\sim$-equivalence class $[\alpha]_\sim$
appears as an $\sim$-equivalence class in $S(G)$ (this is why  the
system is called {\em complete}).

The functor $S$ has a natural dual, namely given $S\models\Sigma$, and
for $\alpha\in S$, letting 
$[\alpha]_\sim$ denote the $\sim$-equivalence class of $\alpha$,
define $$G(S)=\varprojlim_{\alpha\in S}[\alpha]_\sim,$$
where if $\alpha\leq\beta\in S$, the graph of the epimorphism
$\pi_{\alpha,\beta}: [\alpha]_\sim\to [\beta]_\sim$ is given by $C\cap
([\alpha]_\sim\times[\beta]_\sim)$. 

If $A\subset S(G)$, there is a notion of {\em complete subsystem generated by
$A$}: it is the smallest subset $S$ of $S(G)$ satisfying the following
properties:  
if $\alpha\leq\beta$, and $\alpha\in S$ then $\beta\in S$; $1\in S$; if
$\alpha,\beta\in S$ then there is $\gamma\leq\alpha,\beta$ such that
$\gamma\in S$ (note that if $\alpha=gN$ and $\beta=hM$, then
$\gamma=N\cap M$ works). (In other words, the smallest model of $\Sigma$
containing $A$.) Then $G(S(A))$ is the quotient of $G$ by the
intersection of all open
normal subgroups $N$ of $G$ such that some coset $gN$ is in $A$. 

\begin{fact}\begin{enumerate}
\item A profinite group $G$ is small if and only if $S(G)_i$ is finite
  for every $i\geq 1$.

  \item There is a theory $T_{\rm IP}$ which axiomatizes the structures
    $S(G)$ where $G$ has the embedding property.
     There is a theory $T_{\rm Proj}$ which axiomatizes the structures
      $S(G)$ with $G$ projective. Both assertions are essentially
     trivial, since the image of the morphism completing the diagram in
      Definitions \ref{proj-EP}(1) and (3)  has
     size bounded by $|B|$. 

\item Given a sentence $\theta$ of the language of complete systems, and
  a field $K$, 
  there is an $\mathcal{L}(K)$-sentence $\theta^*$ such that $$K\models
  \theta^* \iff S(\gal(K))\models \theta.$$   
(See e.g. \cite{C-omega} for details.) 
 \item Things are a little more complicated for formulas, since $K$ only
   interprets $S(\gal(K))$ ``up to conjugation'', so some care needs to
   be taken. One can reduce to the case of a formula
   $\theta(\xi_1,\ldots,\xi_m)$ implying $\xi_i\sim\xi_j$ for all
   $i\neq j$. 
   
\end{enumerate}
  \end{fact}

\begin{fact} \vlabel{embedding cover facts}
    Let $G$ be a profinite group.
    \begin{enumerate}
        \item The group $G$ has a cover $H$ with the embedding property
          such that $\mathrm{rank}(G) = \mathrm{rank}(H)$ ($=$ minimal
          number of topological generators of $G$). If $G$ is finite and $\mathrm{rank}(G) = e$, then $H$ can be chosen with $|H| \leq |G|^{|G|^{e}}$ (\cite[Corollary~24.3.4]{FJ}).
        \item If $G$ has the embedding property, then so does its
          universal Frattini cover $\tilde{G}$ and, hence, $\tilde{G}$
          is superprojective (\cite[Prop.~24.3.5]{FJ}).
    \end{enumerate}
\end{fact}

\section{Measures in a special case}


In this section, we define measures on definable subsets of perfect PAC fields whose absolute Galois group is the universal Frattini cover of a finite group. This is a very restrictive condition, but we show in the next section that this construction can be leveraged to define measures for \emph{all} bounded perfect PAC fields, and perfect Frobenius fields as well.

\para{\bf Setting} \vlabel{setting-sec3-1}\\
Let $k$ be a perfect field, with finitely generated absolute Galois
group $\gal(k)=\langle \si_1,\ldots,\si_n\rangle$; we also assume that
$k\subset\calk$, where $\calk$
  is a sufficiently saturated perfect PAC field, with the restriction
  map $\gal(\calk)\to \gal(k)$ an isomorphism.

We suppose $V$ is an absolutely irreducible variety defined over $k$ with generic point $a$. We are ultimately interested in defining a measure $\mu_{V}$ on definable subsets of $V$.  

Let $L/k(a)$ be a finite Galois extension, define $k_L=k^s\cap L$, and 
let ${\bar\si}'=(\si_1',\ldots,\si'_n)\in\gal(L/k(a))^n$ be a lift of
$(\si_1,\ldots,\si_n)\rest_{k_L}$. 
Given an intermediate field $k(a) \subseteq
K \subseteq L$ such that $K\cap k_L=k$, i.e., such that $K/k(a)$ is
  regular over $k$, 
 we define $$
\mathcal{S}(L/K) = \left\{F : K \subseteq F \subseteq L, F \cap k_L =
k \right\}.$$
\para\vlabel{first measure}{\bf Definition of a first measure}. Given any $X \subseteq \mathcal{S}(L/K)$, we define a probability
measure on $\cals(L/K)$  by setting
$$
\mu^{1}_{L/K}(X) = \frac{|\{(\tau_{1}, \ldots, \tau_{n}) \in \mathrm{Gal}(L/Kk_{L})^{n} : \mathrm{Fix}(\sigma'_{1}\tau_{1}, \ldots, \sigma'_{n}\tau_{n}) \in X\}|}{[L : Kk_{L}]^{n}}.  
$$
Note that $\bar\tau\in\gal(L/Kk_L)^n$ ensures that
$\fix(\sigma'_{1}\tau_{1}, \ldots, \sigma'_{n}\tau_{n})\in\cals(L/K)$. 

\medskip
To simplify notation, for a field $F \in \mathcal{S}(L/K)$, we will
write $\mu^{1}_{L/K}(F)$ instead of
$\mu^{1}_{L/K}(\{F\})$. Additionally, given $\overline{\tau} =
(\tau_{1}, \ldots, \tau_{n})$ and $\overline{\sigma}' = (\sigma'_{1},
\ldots, \sigma'_{n})$, we will write $\overline{\sigma}' \cdot
\overline{\tau}$ (or $\overline{\sigma}' 
\overline{\tau}$) for the tuple $(\sigma'_{1}\tau_{1}, \ldots,
\sigma'_{n} \tau_{n})$. Finally, it will sometimes simplify equations to
 write $\mu^{1}_{L/K}(F)$ when $F \not\in \mathcal{S}(L/K)$.  In
this case, we stipulate $\mu^{1}_{L/K}(F) = 0$. 

\begin{lem} \vlabel{maximal}\vlabel{rem2}
Suppose $K \supseteq k(a)$ is regular over $k$ 
and $L/k(a)$ is
a finite Galois extension containing $K$. 
Note that by definition of
$\mu^1_{L/K}$,
if $F\in \cals(L/K)$, then $\mu^1_{L/K}(F)>0$ if and only there is a
tuple $\bar\tau\in \gal(L/Kk_L)$ such that
$\fix({\bar\si}'\cdot\bar\tau)=F$.
\begin{enumerate}
   \item \vlabel{no dependence} The definition of $\mu^1_{L/K}$ does not depend on the choice of
    ${\bar\si}'$.
  
 \item \vlabel{extension} Suppose $M$ is a Galois extension of $k(a)$ which
   contains $L$. If $N \in \mathcal{S}(M/K)$, 
   $F\in \cals(L/(N \cap L))$, and $X=\{N'\in\cals(M/N)\mid N'\cap L=F\}$, then
   $\mu^1_{M/N}(X)=\mu^1_{L/(N \cap L)}(F)$. 
 \item Assume that $k_1\subset L$ is a finite Galois extension of $k$
   such that $L\cap k_1=k$ and $\cals(L/K)=\{k(a)\subset K\subset L\mid
   K\cap k_1=k\}$. Then $\mu^1_{L/K}$ computed with respect to
   $k_1$ or to $k_L$ give the same result, i.e., for $F\in\cals(L/K)$:
   $$\frac{|\{\bar\tau\in \gal(L/Kk_L)^n\mid \fix({\bar\si}'\bar\tau)=F\}|
 }{[L:Kk_L]^n
 }=
  \frac{|\{\bar\tau\in \gal(L/Kk_1)^n\mid \fix({\bar\si}'\bar\tau)=F\}|
 }{[L:Kk_1]^n
  }.$$

\end{enumerate}
We now assume that for some finite Galois extension $k_1$ of $k$,
$\gal(k)$ is the universal Frattini cover of
$\gal(k_1/k)$. Note that by Fact \ref{fact2}(5), if $L/k(a)$ is Galois and
contains $k_1$, then for any $k(a)\leq K\leq L$, we have $K\cap k_1=k$
implies $K\cap k^s=k$. 

\begin{enumerate}
  \setcounter{enumi}{3}
  \item If $F\in\cals(L/K)$ is maximal, then any
  lift of $\bar\si$ to $\gal(L/F)$ generates $\gal(L/F)$, and $\mu^1_{L/K}(F)>0$.
\item Let $F\in\cals(L/K)$. Then $F$ is maximal if and only if  there is an epimorphism
  $\gal(k)\to \gal(L/F)$. 

\item Define a subextension $k(a)\leq F\leq L$ to be {\em permissible}
  if $F\in\cals(L/k(a))$, and in some $\calk'\equiv_k\calk$ containing $k(a)$, one has $L\cap
  \calk'=F$. Then $F$ is permissible if and only if $F$ is maximal in $\cals(L/k(a))$. 
  \end{enumerate}
\end{lem}

\begin{proof} (1) is clear, since any lift ${\bar\si}''$ of $\bar\si\rest_{k_L}$ to
  $\gal(L/K)$ is a translate of ${\bar\si}'$ by an $n$-tuple in
  $\gal(L/Kk_L)$.

  (2) This is clear, since if $\bar\tau\in\gal(L/(L \cap N)k_L)^n$ is such that
  $\fix({\bar\si}'\cdot\bar\tau)=F$, then ${\bar\si}'\cdot\bar\tau$ has
  exactly $[M:Lk_M]^n$ distinct lifts to $\gal(M/Nk_M)$.

  (3) Simple computation.

  (4) Since $F\in\cals(L/K)$, we know that $\gal(L/F)$ projects onto
  $\gal(k_1/k)$, and by maximality of $F$, that $\gal(L/F)$ is a minimal
  subgroup of $\gal(L/K)$ projecting onto $\gal(k_1/k)$; the first
  assertion tells us that $\bar\si\rest_{k_1}$ lifts to an element ${\bar
    \si}''\in\gal(L/F)$, and the second, that $\langle
  {\bar\si}''\rangle=\gal(L/F)$. In particular $\mu^1_{L/K}(F)>0$.

  (5) Our assumption that $F \in \mathcal{S}(L/K)$ means that $F\cap k_1=k$ and thus that the restriction $\mathrm{Gal}(L/F) \to \mathrm{Gal}(k_{1}/k)$ is an epimorphism.  If $F$ is maximal, then $\gal(L/F)$ is minimal among subgroups of $\gal(L/K)$ that project onto $\gal(k_{1}/k)$ and thus that the map $\gal(L/F) \to \gal(k_{1}/k)$ is a Frattini cover.  Hence, by definition of universal Frattini cover, there is an epimorphism $\gal(k) \to \gal(L/F)$.  For the other direction, if there is an epimorphism $\gal(k) \to \gal(L/F)$, then we can lift $\bar\si\rest_{k_1}$ to a set of
  generators of $\gal(L/F)$ and obtain an epimorphism $\gal(k)\to \gal(L/F)$ which composes with the restriction $\gal(L/F) \to \gal(k_{1}/k)$ to give the universal Frattini cover of $\gal(k_{1}/k)$ (by, e.g., Gasch\"utz's Lemma \cite[Lemma 17.7.2]{FJ}). This implies $\gal(L/F) \to \gal(k_{1}/k)$ is Frattini, so $F$ is maximal. 

  (6) The following maps are onto: $\gal(k)\to \gal(L/F)\to \gal(k_L/k)$,
  (the first one by permissibility of $F$, and because $\gal(k)\simeq \gal(\calk')$), and (5) gives us the desired
  equivalence.  
\end{proof}

\subsection*{\bf Defining  the desired measure} As we saw above in item
(5), permissible fields $F\in \cals(L/K)$ have positive
$\mu^1_{L/K}$-measure, but there
may be others. We must therefore change the definition of the measure,
so that it tends to $0$ on  non-permissible members of $\cals(L/K)$.  We do this by defining a Markov chain on $\mathcal{S}(L/K)$.  

If $X$ is a finite set, then a \emph{Markov chain} on $X$ is a process
which moves among the elements of $X$ in the following fashion:  at
position $x \in X$, the next position in the process is chosen according
to a fixed probability distribution $P(x,-)$ on $X$, which depends only
on $x$. We say that $X$ is the \emph{state space} of the Markov chain
and $P(x,y)$ is the \emph{transition probability} from $x$ to $y$. If $X
= \{x_{1}, \ldots, x_{n}\}$, then the $n \times n$ matrix $P = (p_{ij})$
defined by $p_{ij} = P(x_{i},x_{j})$ is the \emph{transition matrix} of the Markov chain. 

In any Markov chain on $X$, there is an associated pre-order in which $x' \geq x$ if $P(x',x) > 0$ (i.e., if $x$ can be reached from $x'$).  This pre-order induces a partial order on the equivalence classes, where $x$ is equivalent to $x'$ if $x \leq x'$ and $x' \leq x$. The elements of the minimal equivalence classes are called \emph{ergodic} and the elements that are not ergodic are called \emph{transitory}.  When the equivalence class of an ergodic element is a singleton, the element is called \emph{absorbing}; equivalently, a state $x$ is absorbing if $P(x,x) = 1$ \cite[Theorem 2.4.2]{KS}.  

\begin{fact} \label{ergodic} \cite[Theorem 3.1.1]{KS}
    In a Markov chain on a finite set, regardless of the starting position, the probability that, after $m$ steps, the process is in an ergodic state tends to $1$ as $m$ tends to $\infty$. 
\end{fact}

\begin{defn} \vlabel{mui} Following Johnson \cite[Definition 5.12]{Jo}, we define for $i\in\mathbb{N}$, the
  following probability measures on elements of $\cals(L/K)$.
  \begin{itemize}
  \item $\mu^0_{L/K}(K)=1$;
  \item $\mu^1_{L/K}$ is as in \ref{first measure};
  \item If $i>0$ and $F\in\cals(L/K)$, then
    $$\mu^{i+1}_{L/K}(F)=\sum_{F'\in\cals(L/K)}\mu^i_{L/K}(F')\mu^1_{L/F'}(F).$$ 
    \item $\mu^{\infty}_{L/K}(F) = \lim_{i} \mu^{i}_{L/K}(F)$.  
    \end{itemize}
\end{defn}

We consider the Markov chain on $\mathcal{S}(L/K)$ where the transition
probability of going from $F$ to $F'$ is given by $\mu^{1}_{L/F}(F')$.
More generally, the probability of getting from $F$ to $F'$ in
at most $i$ steps is $\mu^{i}_{L/F}(F')$.  

\begin{lem} \label{ergodic and maximal}
    In the Markov chain on $S(L/K)$ with transition probability from $F$ to $F'$ given by $\mu^{1}_{L/F}(F')$, the following are equivalent for $F \in \mathcal{S}(L/K)$:
    \begin{enumerate}
        \item $F$ is ergodic.
        \item $F$ is absorbing.
        \item $F$ is a maximal field in $S(L/K)$ (with respect to inclusion). 
    \end{enumerate}
\end{lem}

\begin{proof}
    (3)$\implies$(2)  It is clear from the definition of the measure that, for $F,F' \in \mathcal{S}(L/K)$, $\mu^{1}_{L/F}(F') > 0$ implies $F' \supseteq F$, hence if $F$ is maximal then we have $\mu^{1}_{L/F}(F) = 1$ so $F$ is absorbing.  
    (2)$\implies$(1) If $F$ is absorbing then it is ergodic, by definition.  
    (1)$\implies$(3)  If $\mu^{1}_{L/F}(F') > 0$ and $\mu^{1}_{L/F'}(F) > 0$, then we must have $F = F'$ so, in the preorder associated to this Markov chain, all equivalence classes are singletons. By Lemma \ref{maximal}(4), for any $F \in \mathcal{S}(L/K)$ and maximal $F' \supseteq F$, we have $\mu^{1}_{L/F}(F') > 0$ so the ergodic elements are maximal. 
\end{proof}

\begin{lem}
    The measure $\mu^{\infty}_{L/K}$ concentrates on the maximal elements of $\mathcal{S}(L/K)$.  Moreover, if $F \in \mathcal{S}(L/K)$ is maximal, then $\mu^{\infty}_{L/K}(F) > 0$.  
\end{lem}

\begin{proof}
    By Fact \ref{ergodic}, $\mu^{\infty}_{L/K}$ concentrates on the set of ergodic elements of $\mathcal{S}(L/K)$ and by Lemma \ref{ergodic and maximal}, this coincides with the set of maximal fields in $\mathcal{S}(L/K)$.  Finally, by Lemma \ref{maximal}(4), we have for maximal $F \in \mathcal{S}(L/K)$, 
    $$
    0 < \mu^{1}_{L/K}(F) \leq \mu^{\infty}_{L/K}(F).
    $$
Indeed, note that by definition, the values of all $\mu^i$ are
  non-negative. It follows that if $F$ is maximal, one of the summands
  in the definition of $\mu^{i+1}_{L/K}(F)$ is
  $\mu^i_{L/K}(F)\mu^1_{L/F}(F)=\mu^i_{L/K}(F)$, and that the
  $\mu^i_{L/K}(F)$ form a non-decreasing sequence. 
This completes the proof. 
\end{proof}

\begin{prop} \vlabel{rational}
    The measure $\mu^{\infty}_{L/K}$ takes  only rational values on $\mathcal{S}(L/K)$.  
\end{prop}

\begin{proof}
    Choose an enumeration $F_{1}, \ldots, F_{k}$ of $\mathcal{S}(L/k(a))$ such that $F_{i} \subseteq F_{j}$ implies $i \geq j$.  Thus the maximal elements of $\mathcal{S}(L/k(a))$ form an initial segment of the enumeration and $F_{k} = k(a)$. Let $F_{1}, \ldots, F_{\ell}$ denote the maximal elements in this enumeration. We will let the $k \times k$ matrix $P = (p_{ij})$ be the transition matrix with $p_{ij} = \mu^{1}_{L/F_{i}}(F_{j})$.  Note that for $i,j \leq \ell$, we have $\mu^{1}_{L/F_{i}}(F_{j}) = 0$ if $i \neq j$ and $\mu^{1}_{L/F_{i}}(F_{j}) = 1$ if $i = j$. Moreover, if $i \leq \ell$ and $j > \ell$, then we have $\mu^{1}_{L/F_{i}}(F_{j}) = 0$. Thus, we can express the matrix $P$ as 
$$
P = \left(
\begin{matrix}
    I & 0 \\
    R & Q
\end{matrix}
\right)
$$
where $I$ is the $\ell \times \ell$ identity matrix, $Q$ is a square $(k-\ell) \times (k - \ell)$ matrix, and $R$ is an $\ell \times (k-\ell)$ matrix. By \cite[Theorem 3.2.1]{KS}, the matrix $(I - Q)$ is invertible (one can also see this directly: from our choice of enumeration, $Q$ is a lower-triangular matrix with entries $<1$ along the diagonal, hence $(I-Q)$ is lower triangular with non-zero diagonal entries).  We write $N$ for the matrix $(I - Q)^{-1}$.  

Then by \cite[Theorem 3.3.7]{KS}, the matrix $B = (b_{ij})$, where $b_{ij}$ is the probability that the random process starts at a non-maximal field $F_{i}$ and ends at the maximal field $F_{j}$ (i.e. $b_{ij} = \mu^{\infty}_{L/F_{i}}(F_{j})$), is given by the equation 
$$
B = NR.
$$
It is immediate from the definition of the measure that all entries in $P$, hence in $N$ and $R$ are rational numbers.  Thus each entry in $B$ is rational as well.  We obtain, in particular, that $\mu^{\infty}_{L/k(a)}(F) \in \mathbb{Q}$ for every maximal $F \in \mathcal{S}(L/k(a))$.  Since $\mu^{\infty}_{L/k(a)}(F) = 0$ for all non-maximal $F$, we have proved that $\mu^{\infty}_{L/k(a)}$ takes on only rational values. 
\end{proof}

\begin{lem} \vlabel{infinity extension}
    If $M$ is a Galois extension of $k(a)$ which contains $L$, and
   $F\in \cals(L/K)$, $X=\{N\in\cals(M/K)\mid N\cap L=F\}$, then
   $\mu^{\infty}_{M/K}(X)=\mu^{\infty}_{L/K}(F)$. 
\end{lem}

\begin{proof}
    Let $r: \mathcal{S}(M/K) \to \mathcal{S}(L /K)$ denote the surjective map $N \mapsto N \cap L$. We prove $\mu^{i}_{M/K}(r^{-1}(\{F\}))=\mu^{i}_{L/K}(F)$ by induction on $i$.  The $i = 1$ case is Lemma \ref{maximal}(2). Assume for induction that it has been shown for some $i \geq 1$.  Note that $F' \in \mathcal{S}(L/K)$ and $N_{0},N_{1} \in \mathcal{S}(M/K)$ satisfy $L \cap N_{0} = L \cap N_{1}$, then 
$$
\mu^{1}_{M/N_{0}}(r^{-1}(\{F'\})) = \mu^{1}_{L/L \cap N_{0}}(F') = \mu^{1}_{L/L \cap N_{1}}(F') = \mu^{1}_{M/N_{1}}(r^{-1}(\{F'\})).
$$
Then we have 
\begin{eqnarray*}
\mu^{i+1}_{L / K}(F) &=& \sum_{F' \in \mathcal{S}(L/K)} \mu^{i}_{L/ K}(F')\mu^{1}_{L/F'}(F) \\
&=& \sum_{F' \in \mathcal{S}(L/K)} \mu^{i}_{M/K}(r^{-1}(\{F'\}))\mu^{1}_{L/F'}(\{F\}) \\ 
&=& \sum_{F' \in \mathcal{S}(L/K)}\; \sum_{N \in r^{-1}(\{F'\})} \mu^{i}_{M/K}(N)\mu^{1}_{M/N}(r^{-1}(\{F\}) )\\
&=& \sum_{N \in \mathcal{S}(M/K)} \mu^{i}_{M/K}(N)\mu^{1}_{M/N}(r^{-1}(\{F\}) ) \\
&=& \sum_{N \in r^{-1}(\{F\})} \mu^{i+1}_{M/K}(N) \\
&=& \mu^{i+1}_{M/K}(r^{-1}(\{F\})).  
\end{eqnarray*}
Then, taking limits, we obtain $\mu^{\infty}_{M/K}(r^{-1}(\{F\})) = \mu^{\infty}_{L/K}(F)$, as desired. 
\end{proof}

\subsection*{The measure on $V$}

Suppose $\varphi(x)$ is a test formula over $k$ and $L/k(a)$ is a finite Galois extension containing $k_{1}$. We say that $L$ is \emph{adequate} for $\varphi(x)$ if $\mathcal{K} \supseteq k(a)$ is a PAC field with $\mathrm{res}: \gal(\mathcal{K}) \to \gal(k)$ an isomorphism, then whether or not $\mathcal{K} \models \varphi(a)$ depends only on $\mathcal{K} \cap L$. There always is some finite Galois $L/k(a)$ which is adequate for $\varphi(x)$; for example, one can take $L$ to be the splitting field of the polynomials that appear in the test sentence $\varphi(a)$. 

In the previous subsection, we defined the measures $\mu^{\infty}_{L/K}$ on $\mathcal{S}(L/K)$.  Now we use these measures to define a Keisler measure $\mu_{V}$ on definable subsets of $V$.  By Corollary \ref{types}, it suffices to define $\mu_{V}(\varphi(x))$ where $\varphi(x)$ is a test formula that defines a subset of $V$. There is a finite Galois extension $L/k(a)$ containing $k_{1}$ which is adequate for $\varphi(x)$ and therefore also a set $X^{L}_{\varphi} \subseteq \mathcal{S}(L/k(a))$ such that a PAC field $\mathcal{K} \supseteq k(a)$, with $\mathrm{res}: \gal(\mathcal{K}) \to \gal(k)$ an isomorphism, will satisfy 
$$
\mathcal{K} \models \varphi(a) \iff \mathcal{K} \cap L \in X^{L}_{\varphi}.
$$
In other words, we choose $L$ such that the truth value of $\varphi(a)$ in $\mathcal{K}$ depends only on $L \cap \mathcal{K}$ and take $X^{L}_{\varphi}$ to be the set of possibilities for $\mathcal{K} \cap L$ in which $\varphi(a)$ is true in $\mathcal{K}$.  Then we set $\mu_{V}(\varphi(x)) = \mu^{\infty}_{L/k(a)}(X^{L}_{\varphi})$. 

In order to show that the measure $\mu_{V}$ is well-defined, we need to prove that the value $\mu_{V}(\varphi(x))$ does not depend on the choice of $L$:

\begin{lem}
    Suppose $L$ and $M$ are finite Galois extensions of $k(a)$ containing $k_{1}$ which are adequate for the test formula $\varphi(x)$ over $k$. Then 
    $$
    \mu^{\infty}_{L/k(a)}(X^{L}_{\varphi}) = \mu^{\infty}_{M/k(a)}(X^{M}_{\varphi})
    $$
    with $X^{L}_{\varphi}$ and $X^{M}_{\varphi}$ defined as above. 
\end{lem}

\begin{proof} It suffices to prove this in the case that $L \subseteq M$. Unraveling definitions, we have 
    $$
    X^{L}_{\varphi} = \{F \cap L : F \in X^{M}_{\varphi}\}
    $$
    so we have $\mu^{\infty}_{L/k(a)}(X^{L}_{\varphi}) = \mu^{\infty}_{M/k(a)}(X^{M}_{\varphi})$, by Lemma \ref{infinity extension}.
\end{proof}


\section{Perfect bounded PAC fields and perfect Frobenius fields}

In this section, we turn our attention to the construction of measures on arbitrary bounded perfect PAC fields, and also perfect Frobenius fields. As a first step, we show, using the model theory of the inverse system of their absolute Galois groups, that these fields are elementarily equivalent to ultraproducts of perfect PAC fields whose absolute Galois groups are universal Frattini covers of a finite groups.  Then we obtain measures on definable sets via ultralimit measures, using the construction from the previous section.

\begin{lem} \vlabel{bounded ultraproduct}
Suppose $G$ is a bounded projective group.  Then $S(G)$
is an ultraproduct of inverse systems $S(\tilde{G_{i}})$, where
{each} $\tilde{G_{i}}$ is the universal Frattini cover of a finite group. 
\end{lem}

\begin{proof} For each $i$, let $S_i$ be the complete subsystem of $S(G)$ generated by the
  elements of $S(G)$ of sort $\leq i$,  let $G_i=G(S_i)$, $\pi_i:G\to G_i$ the map dual to the
  inclusion $S_i\subset S(G)$, and let $\tau_i: \tilde G_i\to G_i$ be the universal Frattini
  cover of $G_i$. Note that trivially $S(G)=\bigcup S(G_i)$.\\[0.05in]
{\bf Claim}. The elements of sort $\leq i$ of $S(\tilde G_i)$
  are exactly the elements of sort $i$ of $S(G)$. 

\smallskip\noindent{\em Proof of the claim}.  Since $G$ is
  projective, we know that there is an onto map $\rho_i: G\to \tilde
  G_i$ such that $\tau_i\circ \rho_i=\pi_i$. Hence $S(\tilde G_i)$
  embeds into $S(G)$, and this proves the claim. \qed
  
Note that it also shows that for $j\geq i$, the elements of $S(\tilde
  G_j)_{\leq i}$ are exactly the elements of sort $\leq i$ of $S(G_i)$, since
  $S(G_i)\leq S(G_j)$.
Hence, if $\mathcal{D}$ is a non-principal ultrafilter on $\mathbb{N}^*$, and $$ S^{*} = \prod_{i \in {\mathbb N}^*}
S(\tilde{G}_{i})/\mathcal{D},$$ then for all $i$, for all $j\geq i$,
$S(\tilde G_j)_{\leq i}=S(G)_{\leq i}$, and therefore the same is true in
$S^*$: $S^*_{\leq i} \simeq S(G)_{\leq i}$. This shows that
$G(S^*)\simeq G$, and finishes the proof. \end{proof}

{\begin{cor} \vlabel{bounded PAC cor} Let $\calk$ be a bounded perfect PAC field. Then $\calk$ is
  elementarily equivalent to an ultraproduct of perfect PAC fields
  $\calk_i$, with $\gal(\calk_i)$ the universal Frattini cover of a
  finite group (use \cite[Thm 30.6.3]{FJ}).
  \end{cor} }


\begin{lem} \vlabel{pseudo small}
   Suppose $G$ is a superprojective profinite group.  Then $S(G)$ is
   elementarily equivalent to an ultraproduct $\prod_{i \in \omega}
   S(\tilde{G}_{i})/\cald$ where each $\tilde{G}_{i}$ is the universal
   Frattini cover of a finite group, and which we can take with the
   embedding property. 
\end{lem}

\begin{proof}
In the case that $G$ is bounded, this is Lemma \ref{bounded
  ultraproduct}. By the downward L\"owenheim-Skolem Theorem, we may
assume $S(G)$ is countable and choose an increasing sequence $S_{0}
\subseteq S_{1} \subseteq \ldots \subseteq S(G)$ of finite substructures
of $S(G)$ with $\bigcup S_{n} = S(G)$. By dualizing, this corresponds to a sequence of finite groups $G(S_{i}) \in \mathrm{Im}(G)$ with epimorphisms 
$$
G(S_{0}) \twoheadleftarrow G(S_{1}) \twoheadleftarrow G(S_{2}) \twoheadleftarrow \ldots
$$
whose inverse limit is $G$. Let $G_{i}$ be the (finite) embedding cover of $G(S_{i})$ (which exists by Fact \ref{embedding cover facts}(1)) and then let $\tilde{G}_{i}$ be the universal Frattini cover of $G_{i}$. Note that $\tilde{G}_{i}$ is superprojective by Fact \ref{embedding cover facts}(2). Let $\mathcal{D}$ be a non-principal ultraproduct on $\omega$ and let $S_{*} = \prod_{i \in \omega} S(\tilde{G}_{i})/\mathcal{D}$. 

We know $S_{*} \models T_{\rm IP}\cup T_{\rm Proj}$ (i.e. $G(S_{*})$ is a superprojective profinite group) so, to conclude, it suffices to prove that $\mathrm{Im}(G(S_{*})) = \mathrm{Im}(G)$.  Note that, by the choice of the $S_{i}$'s, if $A \in \mathrm{Im}(G)$, then there is some $i$ such that $A \in \mathrm{Im}(G_{i})$ and hence in $\mathrm{Im}(\tilde{G}_{j})$ for all $j \geq i$. This entails $A \in \mathrm{Im}(G(S_{*}))$. On the other hand, if $A \in \mathrm{Im}(G(S_{*}))$, there is a set $X \in \mathcal{D}$ such that $i \in X$ implies $A \in \mathrm{Im}(\tilde{G}_{i})$.  But since $G$ is projective and $G_{i}$ is an image of $G$, $\tilde{G}_{i}$ is also a quotient $G$ so we have $A \in \mathrm{Im}(G)$.  Hence $S_{*} \equiv S(G)$. 
\end{proof}

\begin{cor} Let $\calk$ be a perfect Frobenius field. Then $\calk$ is
  elementarily equivalent to an ultraproduct of perfect Frobenius fields
  $\calk_i$, with $\gal(\calk_i)$ the universal Frattini cover of a
  finite group with the embedding property (as in Corollary \ref{bounded PAC cor}, use \cite[Thm 30.6.3]{FJ}).
  \end{cor}

As a corollary we get the following result: 

\begin{thm}\vlabel{thm1} Let $\calk$ be a perfect PAC field, and $k$ a relatively
  algebraically closed subfield of $\calk$. Let $V$ be an absolutely
  irreducible variety defined over $k$. In each of the following cases,
  there is a probability measure $\mu_V$ defined on the definable
  subsets of $V(\calk)$:
  \begin{itemize}
  \item[(a)] $\gal(k)$ is the Frattini cover of some finite group, and
    the restriction map $\gal(\calk)\to\gal(k)$ is an isomorphism.
  \item[(b)] $\gal(k)$ is small, and the restriction map $\gal(\calk)\to
    \gal(k)$ is an isomorphism.
  \item[(c)] $\calk$ is Frobenius.
    \end{itemize}
  \end{thm}

\begin{proof} In all three cases, we know that definable subsets of
  $V(\calk)$ are defined by test formulas (see Lemmas \ref{lem:test} and
  \ref{test:Frob}).
  Case (a) was already done in Section 3.   

In cases (b) and (c), we know that
$\calk\equiv_k\prod_{i\in\omega}\calk_i/\cald$, where each $\calk_i$ is
perfect PAC, with $\gal(\calk_i)$ 
the {universal} Frattini cover of some finite group $G_i$, and again, we may assume
that $G_i$ projects onto $\gal(k^s\cap L/k)$. Then
$$\mu_V(\theta)=\lim_\cald \mu_{V,i}(X)$$
where $\mu_{V,i}(X)$ is computed in $\calk_i$.
\end{proof}

\begin{prop} \vlabel{efree} Let $k$ be an $e$-free perfect PAC field ($e\in\nat{^>0}$), $V$ an absolutely
  irreducible variety defined over $k$. Then the probability 
  measure $\mu_V$ defined on $V$ in \cite{ChRa} coincides with the
  measure defined in Theorem \ref{thm1}, (see also Corollary \ref{bounded PAC cor}). 
\end{prop}

\begin{proof}
Write $\hat F_e=\gal(k)=\varprojlim_i \tilde G_i$, where each $\tilde G_i$ is the
universal Frattini cover of a finite group $G_i$, and for $i\leq j$, we
have an epimorphism $G_j \twoheadrightarrow G_i$, which induces an
epimorphism $\tilde G_j\twoheadrightarrow \tilde G_i$. For each $i$,
consider a perfect PAC field $k_i$ with $\gal(
k_i)\simeq \tilde G_i$. We saw that if $\cald$ is a non-principal
ultrafilter over $\nat$, then $\calk := \prod k_i/\cald\equiv k$. As each
$G_i$ is a quotient of $\hat F_e$, so is $\tilde G_i$, and by
projectivity, $\tilde G_i$ lifts to a closed subgroup $H_i$ of
$\gal(k)$: we may therefore assume that each $k_i$ contains $k$, so
that our variety is defined over each $k_i$.  

Let $a=(a_i)_i$ be a generic of $V$ over $\calk$, where each $a_i$ is a
generic of $V$ over $k_i$. Let  {$\theta(x)$} be a test formula (with
parameters in $\calk$), and let $L$ be a finite Galois extension of $\mathcal{K}(a)$ which is
adequate for {$\theta(a)$}.  Note that, by {\L}os's Theorem and standard facts, there are finite Galois $L_{i}/k_{i}(a_{i})$ with $\prod L_{i}/\mathcal{D} = L$ and $\gal(L_{i}/k_{i}(a_{i})) \cong \gal(L/\mathcal{K}(a))$ for $\mathcal{D}$-almost all $i$. Moreover, if $F$ is a field, regular over $\mathcal{K}$ with $\mathcal{K}(a) \subseteq F \subseteq L$, then there are $F_{i} \in \mathcal{S}(L_{i}/k_{i}(a_{i}))$.  And, conversely, given $F_{i} \in \mathcal{S}(L_{i}/k_{i}(a_{i}))$ for $\mathcal{D}$-almost all $i$, $\prod F_{i}/\mathcal{D}$ satisfies $\mathcal{K}(a) \subseteq F \subseteq L$ and $F$ is regular over $\mathcal{K}$. Assuming that $L$ contains the splitting field of the polynomials in the test sentence $\varphi(a)$, then we have also that $L_{i}$ is adequate for $\varphi(x)$ (for $\mathcal{D}$-almost all $i$) and $F \in X^{\varphi}_{L}$ if and only if $F = \prod F_{i}/\mathcal{D}$ for $F_{i} \in X^{\varphi}_{L_{i}}$ and the cardinality of $X^{L_{i}}_{\varphi}$ agrees with $|X^{L}_{\varphi}|$ on a $\mathcal{D}$-large set.   

\textbf{Claim}:  Suppose $\mathcal{K}(a) \subseteq F \subseteq L$ and $F$ is regular over $\mathcal{K}$.  Then $F$ is maximal with these properties if and only if $\gal(L/F)$ is $e$-generated if and only if, writing $F = \prod F_{i}/\mathcal{D}$ with $F_{i} \in \mathcal{S}(L_{i}/k_{i}(a_{i}))$, $\mu^{1}_{L_{i}/k_{i}(a_{i})}(F_{i}) \geq \epsilon > 0$ for some $\epsilon > 0$ for $\mathcal{D}$-almost all $i$. 

\emph{Proof of Claim}:  Note that if $F$ is maximal and $F = \prod_{i} F_{i}/\mathcal{D}$, then $F_{i}$ is maximal in $\mathcal{S}(L_{i}/k_{i}(a_{i}))$ (else, an ultraproduct of proper extensions $F_{i} \subsetneq F'_{i} \in \mathcal{S}(L_{i}/k_{i}(a_{i}))$ would produce a witness to the non-maximality of $F$). Then $\mu^{1}_{L_{i}/k_{i}(a_{i})}(F_{i}) \geq \frac{1}{[L_{i} : k_{L_{i}}(a_{i})]}$ by Lemma \ref{maximal}(4) and the definition of the measure. As $\gal(L_{i}/k_{i}(a_{i})) \cong \gal(L/\mathcal{K}(a))$ for $\mathcal{D}$-almost all $i$, this shows $\mu^{1}_{L_{i}/k_{i}(a_{i})}(F_{i})$ is bounded away from $0$ on a $\mathcal{D}$-large set. 

Next, suppose $\mu^{1}_{L_{i}/k_{i}(a_{i})}(F_{i})$ is bounded away from $0$ on a $\mathcal{D}$-large set.  Then for each $i$ in this set, $\gal(L_{i}/F_{i})$ is generated by a lift of the generators of $G_{i}$ and therefore is $e$-generated. It follows that $\gal(L/F)$ is $e$-generated, so there is an epimorphism $\hat{F_{e}} \cong \gal(\calk) \to \gal(L/F)$. As in the proof of Lemma \ref{maximal}(5), this entails that the restriction epimorphism $\gal(L/F) \to \gal(k_{L}/k)$ is Frattini so $F$ is maximal.  \qed

In particular, the above claim shows that if $\mathcal{K}(a) \subseteq F \subseteq L$ and $F$ is regular over $\mathcal{K}$ and maximal with these properties, then, writing $F = \prod F_{i}/\mathcal{D}$ with $F_{i} \in \mathcal{S}(L_{i}/k_{i}(a_{i}))$, we have $\mu^{1}_{L_{i}/k_{i}(a_{i}))}(F_{i}) = \mu^{\infty}_{L_{i}/k_{i}(a_{i})}(F_{i})$ for $\mathcal{D}$-almost all $i$.  

Write $\mu^{*}_{V}$ for the measure constructed in \cite{ChRa} and $\overline{\sigma}$ for a choice of a lift of the generators of $\gal(\calk)$ to $\gal(L/\mathcal{K}(a))$ (and, likewise, $\overline{\sigma}_{i}$ for lifts of generators of $\gal(k_{i})$ to $\gal(L_{i}/k_{i}(a))$). Then, we have
\begin{eqnarray*}
    \mu^{*}_{V}(\varphi(x)) &=& \frac{|\{\overline{\tau} \in \gal(L/\mathcal{K}_{L}(a))^{e} : \mathrm{Fix}(\overline{\tau} \circ \overline{\sigma}) \in X^{\varphi}_{L}\}|}{[L : k_{L}(a)]^{e}} \\
    &=& \lim_{\mathcal{D}} \frac{|\{\overline{\tau} \in \gal(L_{i}/k_{L_{i}}(a_{i}))^{e} : \mathrm{Fix}(\overline{\tau} \circ \overline{\sigma_{i}}) \in X^{\varphi}_{L_{i}}\}|}{[L_{i} : k_{L_{i}}(a)]^{e}} \\
    &=& \lim_{\cald} \mu^{1}_{L_{i}/k_{i}(a_{i})}(X^{\varphi}_{L_{i}}) \\
    &=& \lim_{\cald} \mu^{\infty}_{L_{i}/k_{i}(a_{i})}(X^{\varphi}_{L_{i}}) \\
    &=& \mu_{V}(\varphi(x)). 
\end{eqnarray*}
This shows that the measure of \cite{ChRa} agrees with the measure constructed above. \end{proof}

\begin{thm}\vlabel{thm2}
Let $\calk$ be a perfect PAC field, and $k$ a relatively
  algebraically closed subfield of $K$. Let $G$ be a group definable (over
  $k$) in $\calk$. In each of the following cases, $G$ is definably amenable:
  \begin{itemize}
  \item[(a)] $\gal(k)$ is the Frattini cover of some finite group, and
    the restriction map $\gal(\mathcal{K})\to\gal(k)$ is an isomorphism.
  \item[(b)] $\gal(k)$ is small, and the restriction map $\gal(K)\to
    \gal(k)$ is an isomorphism.
  \item[(c)] $\calk$ is perfect Frobenius.
    \end{itemize}
\end{thm}

\begin{proof} Case (a):  Reason as in \cite{ChRa}: we already know that if $G$ is
  a (connected) algebraic group, then $\mu_G$ is stable under
  translation (\cite[Prop.~3.15]{ChRa}), so that $\gal(\calk)$ is definably amenable. Further, by Theorem C in
  \cite{HrP}, an arbitrary group $G$ definable in $\calk$ is virtually
  isogenous to an algebraic group: there are an algebraic group $H$
  defined over $k$, a definable subgroup $G_0$
  of $G$ of
  finite index, and a definable homomorphism $f:G_0\to H(\calk)$ with
  finite kernel, and with $f(G_0)$ of finite index in $H(\calk)$. The
definable  amenability of $H(\calk)$ then implies easily the definable amenability of $G$
  (see  \cite[Lemma 3.16]{ChRa}).

In cases (b) and (c), we know that
$\calk\equiv_k\prod_{i\in\omega}\calk_i/\cald$, with $\calk_i$ as in
(a). Thus, if $G$ is defined by $\varphi(x)$ in $\calk$, the formula
$\varphi(x)$ defines a group $G_i$ in (almost all in the sense of
$\cald$) $\calk_i$, and each $G_i$ is definably amenable; hence so is $G$.
\end{proof}
 
  \begin{rem} \label{rem: no classification}
More broadly, our results imply that groups definable in a perfect PAC
field that is elementarily equivalent to an ultraproduct of perfect PAC fields whose Galois groups are the universal
Frattini covers of finite groups, will be definably amenable.  Such a
PAC field can be equivalently described as one whose Galois group is
elementarily equivalent to an ultraproduct of such groups in the inverse
system language.  We have shown that this class of fields contains both the bounded perfect PAC fields and the perfect Frobenius fields, but it also contains more contrived examples. It is easily seen that the CDM graph coding \cite[chapter
28, \S\S 7 -- 10]{FJ} takes pseudo-finite graphs to such groups.  Additionally, every pseudo-finite
structure in a finite language is bi-interpretable with a pseudo-finite
graph. This suggests there is no nice classification for this class of PAC fields.  
  \end{rem}


\subsection*{Fields with an action of a finite group}

Suppose $G$ is a finite group.  The language of $G$-fields consists of
the language of rings together with unary function symbols $\sigma_{g}$
for each element $g \in G$. The theory $T_{0}$ is the theory of integral
domains with an action of $G$.  Hoffmann and Kowalski (\cite[Theorem 2]{HK}) show that $T_{0}$ has a model companion $G$-TCF, whose models are existentially closed fields with an action of $G$; we denote models of $G$-TCF by $(K,\overline{\sigma})$ where $K$ is the underlying field and $\overline{\sigma} = (\sigma_{g}^{K})_{g \in G}$ is the tuple of automorphisms of $K$ that give the interpretations of the symbols $\sigma_{g}$ for each $g \in G$. We write $K^{G}$ for the fixed field. 

\begin{fact}\vlabel{G-TCF}
    Let $(K,\overline{\sigma}) \models G\hbox{-TCF}$ and let $F = K^{G}$. Then we have the following:
    \begin{enumerate}
        \item 
Up to adding finitely many constants for elements of $F$ and of $K$,   
          $\mathrm{Th}(F)$ (in the language of rings) is
          bi-interpretable with
          $\mathrm{Th}(K,\overline{\sigma})$. \cite[Remark 2.3]{HK} 
        \item The field $F$ is perfect, PAC, and the absolute Galois group $\gal(F)$ is the universal Frattini cover of $G$. \cite[Theorem 3.40]{HK} 
    \end{enumerate}
\end{fact}

As before, we get as corollaries the following results:

\begin{thm} Let $(\calk,\overline{\sigma}) \models G\hbox{-TCF}$, and
  let $k$ be a relatively algebraically closed subfield of $\calk$. Let
  $V$ be a variety defined over $k$. Then there is a probability measure
  $\mu_V$ on definable subsets of $V(\calk)$.

  If $H$ is a group definable in $\calk$, then $H$ is definably
  amenable. 
  
\end{thm}

\begin{proof} Clear by Fact \ref{G-TCF} and by Theorems \ref{thm1} and \ref{thm2}.\end{proof}

\subsection*{\bf Concluding Remarks and Questions}

(1) Suppose
$\calk$ is perfect PAC with $\gal(\calk)$ $e$-free pro-$p$, $e\in\mathbb{N}$, let $V$ be a variety
defined over $k\subset\calk$. Then $\gal(\calk)$ is the universal
Frattini cover of $(\mathbb{Z}/p\mathbb{Z})^e$. Does the measure $\mu_V$ computed in Theorem \ref{thm1} coincide with the measure we defined in \cite[Section 4.15]{ChRa}? \\

(2) We know by results in \cite{ChRa} that if $\calk$ is $\omega$-free perfect,
and $V$ is a variety defined over $\calk$, then $\mu_V$ only takes the
values $\{0,1\}$. I.e., the only type of $V$ over some relatively
algebraically closed $k\subset \calk$ over which $V$ is defined, and  with non-zero measure is
the type of an element $a$ generic of $V$ over $k$, and such that
$k(a)^s\cap \calk=k(a)$.

Describe $\mu_V$ when $\calk$ is an arbitrary perfect Frobenius field
(of infinite rank). Which
values can it take? What about definable subgroups of $G(\calk)$, for
$G$ a connected algebraic group? (In the case of perfect $\omega$-free
fields, we know that proper subgroups have measure $0$).\\

(3) What are the perfect PAC fields for which definable groups are
definably amenable? The $G$-TCF example (and our method for defining the
measure) suggests that they have to be
definable in a perfect PAC field in which formulas are equivalent to
test formulas.  {See also Remark \ref{rem: no classification}}.\\

{(4) If $k$ is a bounded perfect PAC field and $V$ is an absolutely
  irreducible variety defined over $k$, must the measure $\mu_{V}$ take
  on only rational values?  We showed, in Proposition \ref{rational},
  that the answer is yes when $\mathrm{Gal}(k)$ is the universal
  Frattini cover of a finite group, or when $\gal(k)$ is free, or
  $e$-free pro-$p$.}

\begin{exmp}\vlabel{rational-2} Here is an example, of a bounded perfect
  PAC field, where the measures $\mu_V$ only take rational values, and  {whose absolute Galois group}
  is not finitely generated. Let
  $(\pi_i)_{i\in\nat}$ be an infinite set of disjoint finite sets of
  prime numbers, and for each $i\in\nat$, let $e_i$ be a positive integer, and
  $G_i$ an $e_i$-generated projective pro-$\pi_i$-group (i.e, the order of
  any finite quotient of $G_i$ is only divisible by elements of
  $\pi_i$). If $\calk$ is a perfect PAC field with $\gal(\calk)\simeq
  \prod_{i\in\nat}G_i$, then $\gal(\calk)$ is projective (since the
  orders of elements of distinct $G_i$ are relatively prime),  and bounded; moreover,
  if $\theta(x,y)$ is an $\call(k)$-formula (where $k\prec\calk$), and $L$
 is a finite Galois extension of
  $k(a)$, there is some $s$ such that whenever $k(a)\subset F\subset L$
 is regular over $k$, then $\gal(L/F)$ is a quotient
  of $\prod_{j\leq s}G_j$; thus $\gal(L/F)$ is finitely generated (by
  $\leq e=\sup_{j\leq s}e_j$ elements), and
  as the computation of $\mu^\infty_{L/k(a)}(K_i)$ only depends on its
  computation within $\prod_{j\leq s}G_j$ (because there is no
  non-trivial morphism $\prod_{j>s}G_j$ to $\gal(L/F)$ for $F$ as
  above), we get that whenever $k(a)\subset F\subset L$ is permissible, then $\mu^\infty_{L/k(a)}(F)\in\rat$, by Proposition \ref{rational}.

\end{exmp}

(5) Since it has not appeared in print, we reproduce the proof, due to Tom Scanlon, that bounded perfect PAC fields which are not pseudo-finite are not pseudo-stable (that is, are not elementarily equivalent to an ultraproduct of stable fields). Let $K$ be such a field. By Ax's Theorem, we know that $\mathrm{Gal}(K) \not\cong \hat{\mathbb{Z}}$ so there must be some $\ell$ such that $K$ has either no Galois extensions of degree $\ell$ or at least $2$ Galois extensions of degree $\ell$. Note that if a finite extension of $K$ is not pseudo-stable then $K$ cannot be either. Thus, possibly replacing $K$ with a finite algebraic extension, we may assume there is some $n$ such that the map $x \mapsto x^{n}$ is not surjective. Let $m$ denote the index of $K^{\times n}$ in $K^{\times}$. Then we may write a sentence $\varphi$ which consists of the axioms of fields, together with the assertions that $[K^{\times} : K^{\times n}] = m$ and $K$ has no extensions of degree $\ell$ or at least $2$ extensions of degree $\ell$.  Any field that satisfies $\varphi$ must be infinite, since any finite field has a unique degree $\ell$ extension.  However, any infinite field satisfying $\varphi$ cannot be stable, since the multiplicative group of an infinite stable field is connected \cite[Th\'eor\`eme 3]{P}.

\subsection*{Acknowledgements}

We would like to thank Igor Pak for helpful correspondence on Markov chains.

\vskip 1cm
\end{document}